\documentclass[12pt, reqno]{amsart}
\usepackage{color}
\usepackage{amsmath}
\usepackage{amssymb}
\usepackage{amsfonts}
\usepackage{mathrsfs}
\usepackage{amsbsy}
\usepackage{relsize}
\usepackage[colorlinks, linkcolor=red, citecolor=blue, urlcolor=blue, pagebackref, hypertexnames=false]{hyperref}
\usepackage[hyperpageref]{backref}

\usepackage{geometry}
\newcommand{\R}{\mathbb R}

\newtheorem{proposition}{Proposition}[section]
\newtheorem{theorem}{Theorem}[section]
\newtheorem{remark}{Remark}[section]
\newtheorem{lemma}{Lemma}[section]

\numberwithin{equation}{section}

\numberwithin{equation}{section}

\let\z=\zeta

\title[Inhomogeneous Parabolic Equations]{Well-posedness and blow-up for an inhomogeneous semilinear parabolic equation}
\author{Mohamed Majdoub}

\address{Department of Mathematics, College of Science, Imam Abdulrahman Bin Faisal University, P. O. Box 1982, Dammam, Saudi Arabia\newline
Basic and Applied Scientific Research Center, Imam Abdulrahman Bin Faisal University, P.O. Box 1982, 31441, Dammam, Saudi Arabia}
\email{mmajdoub@iau.edu.sa}

\begin{document}
\begin{abstract}
       We consider the large-time behavior of sign-changing solutions of the inhomogeneous equation $u_t-\Delta u=|x|^\alpha |u|^{p}+\zeta(t)\,{\mathbf w}(x)$ in  $(0,\infty)\times\R^N$, where $N\geq 3$, $p>1$, $\alpha>-2$, $\z, {\mathbf w}$ are continuous functions such that $\z(t)=t^\sigma$ or $\z(t)\sim t^\sigma$ as $t\to 0$, $\z(t)\sim t^m$ as $t\to\infty$ . We obtain local existence for $\sigma>-1$. We also show the following:
\begin{itemize}
\item If $m\leq 0$, $p<\frac{N-2m+\alpha}{N-2m-2}$ and $\int_{\R^N}{\mathbf w}(x)dx>0$, then all solutions blow up in finite time;
\item If $m> 0$, $p>1$ and $\int_{\R^N}{\mathbf w}(x)dx>0$, then all solutions blow up in finite time;
\item If $\z(t)=t^\sigma$ with $-1<\sigma<0$, then for $u_0:=u(t=0)$ and ${\mathbf w}$ sufficiently small the solution exists globally.
\end{itemize}
We also discuss lower dimensions. The main novelty in this paper is that blow up depends on the behavior of $\z$ at infinity.

\end{abstract}



\maketitle
\date{\today}


\section{Introduction}
We study the global existence and blow up of solutions of the following
semilinear parabolic Cauchy problem
\begin{equation}
\label{main}
\left\{
\begin{matrix}
u_t-\Delta u=|x|^\alpha |u|^{p}+\z(t)\,{\mathbf w}(x)\quad \mbox{in}\quad (0,\infty)\times\R^N,\\
u(0,x)= u_0(x)\qquad\mbox{in}\quad\R^N,\\
\end{matrix}
\right.
\end{equation}
where $N\geq 3$, $\alpha\in\R$, $p>1$ and $\z, {\mathbf w}$ are given functions. More specific assumptions on ${\mathbf w}$, $\z$ and $u_0$ will be made later. Our model \eqref{main} arises in many physical phenomena and biological species theories, such as the concentration of diffusion of some fluid, the density of some biological species, and heat conduction phenomena, see \cite{Hu, GV, QS, MT1, MT2} and references therein.

 We are interested in finding the critical exponent which separates the existence and nonexistence of global solutions of \eqref{main}. In the case $\z\equiv 0$ or ${\mathbf w}\equiv 0$, problem \eqref{main} reduces to
\begin{equation}
\label{Fuj}
\left\{
\begin{matrix}
u_t-\Delta u=|x|^\alpha |u|^{p} \quad \mbox{in}\quad (0,\infty)\times\R^N,\\
u(0,x)= u_0(x)\qquad\mbox{in}\quad\R^N.\\
\end{matrix}
\right.
\end{equation}
For nonnegative initial data, the solution of \eqref{Fuj} blows up in finite time if $u_0$ is sufficiently large. For arbitrary initial data $u_0\geq 0$, we have the following dichotomy:
\begin{itemize}
\item if $\alpha >-2$ and $1<p\leq 1+\frac{2+\alpha}{N}$, then every nontrivial solution $u(t,x)$ blows up in finite time;
\item if $\alpha >-2$, $p> 1+\frac{2+\alpha}{N}$ and $u_0$ is sufficiently small, then $u(t,x)$ is a global solution
\end{itemize}

This result was proved by Fujita in \cite{fujita} for $\alpha=0$, $p\neq 1+\frac{2}{N}$, and by Hayakawa in \cite{Hayak} for $\alpha=0$, $p=1+\frac{2}{N}$. Later, Qi in \cite{Qi} was able to prove similar results for a wide class of parabolic problems including in particular \eqref{Fuj}. See also \cite{LM}. The number $p_F:=1+\frac{2+\alpha}{N}$ is called the critical Fujita exponent for problem \eqref{Fuj}. \\
Note that the case $\alpha=0,\, \z\equiv 1$ was investigated in \cite{BLZ}. It was shown, among other results, that \eqref{main} has no global solutions provided that $p<\frac{N}{N-2}$ and $\int_{\R^N}\,{\mathbf w}(x)\,dx>0$. Recently in \cite{JKS}, the authors consider \eqref{main} with $\alpha=0$ and $\z(t)=t^\sigma$ where $\sigma >-1$. They showed that the critical exponent is given by
\begin{eqnarray*}
 p^*(\sigma)&=&\; \left\{
\begin{array}{cllll}\frac{N-2\sigma}{N-2\sigma-2}
\quad&\mbox{if}&\quad -1<\sigma<0,\\\\ \infty \quad
&\mbox{if}&\quad \sigma>0.
\end{array}
\right.
\end{eqnarray*}
We refer the interested reader to the survey papers \cite{DL, Le}. See also \cite{AG, EG, Gia, Stu, Zh1, Zh2} for related problems. In particular, the chemical reaction diffusion equation with special diffusion coefficient $D=|x|^2$ has been extensively studied  in \cite{AG, Gia, Hu, Tan}. Roughly speaking, this can be used to motivate the restriction $\alpha>-2$ which is made in some cases.

Our main motivation for the current work comes from the paper \cite{JKS}, where the authors consider only the case where $\alpha=0$ and $\z(t)=t^\sigma$, $\sigma>-1$. We will improve the blow-up results obtained in \cite{JKS} by considering $\alpha>-2$ and allowing that $\z$  behaves like $t^\sigma, \sigma>-1$ as $t\to 0$ and like $t^m, m\in\R$ as $t\to\infty$. For simplicity of presentation, we shall restrict our attention to $\z$ satisfying either
\begin{equation}
\label{zee}
\z(t)=t^\sigma,
\end{equation}
or
\begin{eqnarray}
\label{ze}
 \z(t)&=&\; \left\{
\begin{array}{cllll}t^\sigma
\quad&\mbox{if}&\quad 0<t<1,\\\\ t^m \quad
&\mbox{if}&\quad t\geq 1,
\end{array}
\right.
\end{eqnarray}
where $\sigma>-1$ and $m \in\R$.

As is a standard practice, we study the local well-posedness of \eqref{main} via the associated integral equation:
\begin{equation}
\label{integral}
u(t)= {\rm e}^{t\Delta}u_{0}+\int\limits_{0}^{t}{\rm e}^{(t-\tau)\Delta}\,\left(|\cdot|^{\alpha} |u(\tau)|^{p}\right)\,d\tau+\int\limits_{0}^{t}{\rm e}^{(t-\tau)\Delta}\,\left(\z(\tau)\,{\mathbf w}\right)\,d\tau,
\end{equation}
where ${\rm e}^{t\Delta}$ is the linear heat semi-group. Using fixed point argument in suitable complete metric space together with a recent smoothing estimate proved in \cite{BTW}, we obtain the following existence results.
\begin{theorem}
\label{WP1}
Suppose $N\geq 3$, $-2<\alpha\leq 0$, ${\mathbf w}\in C_B(\R^N):=C(\R^N)\cap L^\infty(\R^N)$ and $\z$ is given by \eqref{ze} or \eqref{zee}. Then, for any $u_0\in C_B(\R^N)$, the Cauchy problem \eqref{main} has a unique maximal $C_B$-mild solution $u$ on $[0,T^*)\times \R^N$ such that if $T^*<\infty$, then $\displaystyle\lim_{t\to T^*}\,\|u(t)\|_{L^\infty(\R^N)}=\infty.$ Furthermore, if $u_0\geq 0$ and ${\mathbf w}\geq 0$, then the solution $u$ is nonnegative.
\end{theorem}
\begin{remark}
For $N\geq 3$ and $-2<\alpha\leq 0$, we have $0\leq -\alpha <N$. Hence, we can apply the smoothing effect given by Proposition \ref{Key}.
In addition, as we will see in the estimate \eqref{Cont1}, the assumption $\alpha>-2$ is crucial to make $T^{1+\alpha/2}$ small for $T>0$ small.
\end{remark}

For $\alpha >0$, we introduce as in \cite{Wang} the function $\nu(x)=\left(1+|x|\right)^{\frac{\alpha}{p-1}}$ and define
\begin{equation}
\label{Cnu}
{C}_{\nu}(\R^N)=\Big\{\, \varphi\in C(\R^N);\;\;\; \|\nu\varphi\|_{L^\infty}<\infty\,\Big\},
\end{equation}
endowed with the norm
\begin{equation}
\label{Nnu}
\|\varphi\|_{\nu}=\|\nu\varphi\|_{L^\infty}.
\end{equation}

\begin{theorem}
\label{WP2}
Suppose $N\geq 1$, $\alpha>0$, ${\mathbf w}\in {C}_{\nu}(\R^N)$ and $\z$ is given by \eqref{ze} or \eqref{zee}. Then, for any $u_0\in { C}_{\nu}(\R^N)$, the Cauchy problem \eqref{main} has a unique maximal classical solution $u$ on $[0,T^*)\times \R^N$ such that if $T^*<\infty$, then $\displaystyle\lim_{t\to T^*}\,\|u(t)\|_{\nu}=\infty.$
\end{theorem}

Concerning blow-up, we suppose that $\z$ is given by \eqref{ze} and we separate the cases $m\leq 0$ and $m>0$ as stated below.

\begin{theorem}
\label{Blow}
Suppose $N\geq 3$, $\alpha>-2$, $m\leq 0$ and $1<p<\frac{N-2m+\alpha}{N-2m-2}$. Assume that $\z$ is given by \eqref{ze} and $\mathbf{w} \in C_0(\R^N)\cap L^1(\R^N)$ obeys $\int_{\R^N}\,{\mathbf w}(x)\,dx>0$. Then the Cauchy problem \eqref{main} has no global solutions.
\end{theorem}
\begin{remark}
If $m\leq 0$ and $N\geq 3$ then $N-2m-2>0$. The condition $1<p<\frac{N-2m+\alpha}{N-2m-2}$ implies that $\frac{N-2m+\alpha}{N-2m-2}>1$. Hence $\alpha>-2$.
\end{remark}
In the next theorem, we remove the restriction $N\geq 3$ for the dimension but only in the case $m>0$.
\begin{theorem}
\label{Bloww}
Suppose $N\geq 1$, $m> 0$ and $p>1$. Assume that $\z$ is given by \eqref{ze} and $\mathbf{w} \in C_0(\R^N)\cap L^1(\R^N)$ obeys $\int_{\R^N}\,{\mathbf w}(x)\,dx>0$. Then the Cauchy problem \eqref{main} has no global solutions.
\end{theorem}

\begin{remark}\quad\\
\vspace{-0.4cm}
\begin{itemize}
\item[(i)] {\rm Unlike to \cite{JKS} where the critical exponent is given in term of $\sigma$, here the critical exponent depends only on $m$ which measures the behavior of $\z$ at infinity. Indeed, the behavior at $0$ and the fact that $\sigma>-1$ are needed only for the local existence.}
\item[(ii)] {\rm The method apply for more general $\z$ by assuming that $\z : (0,\infty)\to (0,\infty)$ is continuous and
\begin{eqnarray*}
 \z(t)&\sim&\; \left\{
\begin{array}{cllll}c_0\,t^\sigma
\quad&\mbox{as}&\quad t\to 0,\\\\ c_\infty\,t^m \quad
&\mbox{as}&\quad t\to\infty,
\end{array}
\right.
\end{eqnarray*}
where $c_0, c_\infty>0$, $\sigma>-1$ and $m\in\R$.}
\item[(iii)]{\rm A similar statement of Theorem \ref{Blow} for lower dimensions $N=1,2$ reads as follows:}\\
\noindent{\bf Theorem.} {\it Suppose $N=1,2$, $\alpha>-2$, $m<\frac{N}{2}-1$ and $1<p<\frac{N-2m+\alpha}{N-2m-2}$. Assume that $\z$ is given by \eqref{ze} and $\mathbf{w} \in C_0(\R^N)\cap L^1(\R^N)$ obeys $\int_{\R^N}\,{\mathbf w}(x)\,dx>0$. Then the Cauchy problem \eqref{main} has no global solutions.}
\end{itemize}
\end{remark}
Moving to the analysis of the global existence, our main result reads as follows.
\begin{theorem}\label{GE}
Let $N\geq 2$, $-2<\alpha<0$ and $\z$ be given by \eqref{zee} with $-1<\sigma<0$. Assume that $p\geq 1+ \frac{2+\alpha}{N-2(\sigma+1)}$ and set $\ell=\frac{N p_c}{N+2(\sigma+1) p_c}=\frac{N(p-1)}{2+\alpha+2(\sigma+1)(p-1)}$ where $p_c=\frac{N(p-1)}{2+\alpha}$.  Then for any $u_0\in L^{p_c}(\mathbb{R}^N)$ and ${\mathbf w}\in L^{\ell}(\mathbb{R}^N)$ with the property that $\|u_0\|_{L^{p_c}}+\|{\mathbf w}\|_{L^{\ell}}$ is sufficiently small,  Eq. \eqref{integral} admits a global-in-time solution $u$.
\end{theorem}

The paper is organized as follows. In Section 2, we recall some preliminaries needed in the sequel such as smoothing effect for the heat semi-group. The third section is devoted to the local existence for \eqref{main}. In the fourth section, we will focus on nonexistence of global solutions by proving Theorem \ref{Blow} and Theorem \ref{Bloww}. Lastly, in Section 5 we present the proof of the global existence result stated in Theorem \ref{GE}. In all this paper, $C$ will be a positive constant which  may have different values at different places.


\section{Preliminaries}

Let ${\rm e}^{t\Delta}$ be the linear heat semi-group defined by ${\rm e}^{t\Delta}\,\varphi=G_t\star\varphi, t>0,$ where $G_t$ is the heat kernel given by
$$
G_t(x)=\left(4\pi t\right)^{-N/2}\,{\rm e}^{-\frac{|x|^2}{4t}},\;\;\;t>0,\;\;\;x\in\R^N.
$$
Let, for $\gamma\geq 0$, ${\mathbf S}_{\gamma}$ be defined as
\begin{equation}
\label{S}
{\mathbf S}_{\gamma}(t)\varphi={\rm e}^{t\Delta}\left(|\cdot|^{-\gamma}\varphi\right),\;\;\;t>0.
\end{equation}
To treat the nonlinear term in \eqref{main}, we use the following key estimate proved in \cite{BTW}.
\begin{proposition}
\label{Key}
Let $N\geq 1$ and $0< \gamma<N$. For $1<q_1, q_2\leq \infty$ such that
\begin{equation}
\label{q12}
\frac{1}{q_2}<\frac{\gamma}{N}+\frac{1}{q_1}<1,
\end{equation}
we have
\begin{equation}
\label{SS}
\|{\mathbf S}_{\gamma}(t)\varphi\|_{q_2}\leq\,C\,t^{-\frac{N}{2}\left(\frac{1}{q_1}-\frac{1}{q_2}\right)-\frac{\gamma}{2}}\,\|\varphi\|_{q_1},
\end{equation}
where $C>0$ is a constant depending only on $N, \gamma, q_1$ and $q_2$.
\end{proposition}
\begin{remark}\quad\\
\vspace{-0.4cm}
\begin{itemize}
\item[(i)] {\rm For $\gamma=0$, the estimate \eqref{SS} holds under the assumption $1\leq q_1\leq q_2\leq\infty$. This is unlike to the case $\gamma>0$, where \eqref{q12} enable us to take $q_2<q_1$.}
\item[(ii)] {\rm As pointed out in \cite{Tay}, we may take $\frac{1}{q_2}=\frac{\gamma}{N}+\frac{1}{q_1}$, $q_1<\infty$, $q_2<\infty$ in \eqref{q12}.}
    \end{itemize}
\end{remark}

The following Lemma will be useful in the proof of Theorem \ref{WP2}.
\begin{lemma}
\label{Lnu}
Let $\gamma, \kappa>0$. There exists a constant $C=C(\kappa, \gamma, N)>0$ such that, for all $x\in\R^N$ and $\lambda\in [0,\kappa]$, we have
\begin{equation}
\label{gl}
\int\limits_{\R^N}\,{\rm e}^{-|z|^2}\,\left(1+|x-\lambda z|\right)^{-\gamma}\,dz\leq C\left(1+|x|\right)^{-\gamma}.
\end{equation}
\end{lemma}
\begin{proof}
Define
\begin{eqnarray*}
{\mathbf A}&=&\left\{\,z\in\R^N;\;\; |x-\lambda z|\leq \frac{|x|}{2}\;\right\},\\
{\mathbf B}&=&\left\{\,z\in\R^N;\;\; |x-\lambda z|> \frac{|x|}{2}\;\right\}.
\end{eqnarray*}
Clearly $\left(1+|x-\lambda z|\right)^{-\gamma}\leq 2^{\gamma}\,\left(1+|x|\right)^{-\gamma}$ for $z\in {\mathbf B}$. Hence
\begin{equation}
\label{Est1}
\int\limits_{{\mathbf B}}\,{\rm e}^{-|z|^2}\,\left(1+|x-\lambda z|\right)^{-\gamma}\,dz\leq 2^{\gamma} \pi^{N/2}\,\left(1+|x|\right)^{-\gamma}.
\end{equation}
For $z\in {\mathbf A}$ we have $|z|\geq \frac{|x|}{2\lambda}$. It follows that
\begin{eqnarray*}
\int\limits_{{\mathbf A}}\,{\rm e}^{-|z|^2}\,\left(1+|x-\lambda z|\right)^{-\gamma}\,dz&\leq&\int\limits_{\{|z|\geq \frac{|x|}{2\lambda}\}}\,{\rm e}^{-|z|^2}\,dz\\
&\leq&|{\mathcal S}^{N-1}|\,\int\limits_{\frac{|x|}{2\lambda}}^\infty\,{\rm e}^{-r^2}\,r^{N-1}\,dr\\
&\leq&|{\mathcal S}^{N-1}|\left(\int\limits_0^\infty\,{\rm e}^{-\frac{r^2}{2}}\,r^{N-1}\,dr\right)\,{\rm e}^{-\frac{|x|^2}{8\lambda^2}}.
\end{eqnarray*}
Observe that when $0\leq\lambda\leq\kappa$,
$$
{\rm e}^{-\frac{|x|^2}{8\lambda^2}}\leq {\rm e}^{-\frac{|x|^2}{8\kappa^2}}\leq C(\kappa,\gamma)\, \left(1+|x|\right)^{-\gamma}.
$$
Hence
\begin{equation}
\label{Est2}
\int\limits_{{\mathbf A}}\,{\rm e}^{-|z|^2}\,\left(1+|x-\lambda z|\right)^{-\gamma}\,dz\,\leq\, C(\kappa,\gamma,N)\,\left(1+|x|\right)^{-\gamma}.
\end{equation}
Combining \eqref{Est1} and \eqref{Est2}, we obtain \eqref{gl} as desired. \qed\medskip
\end{proof}

We also recall the following singular Gronwall inequality proved in \cite{DM1986}.
\begin{proposition}
\label{SGI}
Let $\psi : [0,T]\to [0,\infty)$ be a continuous function satisfying
\begin{equation}
\label{SGI1}
\psi(t)\leq A+M\int\limits_0^t\,\frac{\psi(\tau)}{(t-\tau)^{\theta}}\,d\tau, \quad 0\leq t\leq T,
\end{equation}
where $0\leq \theta<1$ and $A, M\geq 0$ are two constants. Then
\begin{equation}
\label{SGI2}
\psi(t)\leq A\, {\mathcal E}_{1-\theta}\left(M \Gamma(1-\theta)\,t^{1-\theta}\right),\quad 0\leq t\leq T,
\end{equation}
where ${\mathcal E}_{\varrho}$ is the {\tt Mittag-Leffler} function defined for all $\varrho>0$ by
$$
{\mathcal E}_{\varrho}(z)=\sum_{n=0}^{\infty}\,\frac{z^n}{\Gamma(n\varrho+1)}.
$$
\end{proposition}
Finally, we recall a comparison principle of Phragm\`en-Lindel\"of type for \eqref{main}. See for instance \cite[Lemma 1.3, p. 559]{Wang}. In our context, $f(t,x,u)=|x|^\alpha |u|^p+\z(t){\mathbf w}(x)$.
\begin{lemma}
\label{CPL}
Suppose $\overline{u}$ and $\underline{u}$ are continuous weak upper and lower solutions of \eqref{main} and $(\overline{u}-\underline{u})(t,x)\geq -B \exp(\beta |x|^2)$ on $\R^N$ with $B, \beta>0$. Suppose that
\begin{equation}
\label{CPL1}
|x|^\alpha \bigg[|\overline{u}(t,x)|^p-|\underline{u}(t,x)|^p\bigg]\geq \mathsf{C}(t,x)\bigg(\overline{u}-\underline{u}\bigg)(t,x),
\end{equation}
where $\mathsf{C}$ is continuous and $\mathsf{C}(t,x)\leq C_0 \left(1+|x|^2\right)$ for some $C_0>0$. Then $\overline{u}\geq \underline{u}$ on $\R^N$.
\end{lemma}


\section{Local well-posedness}

First we investigate the case $-2<\alpha\leq 0$ as stated in  Theorem \ref{WP1}.
\begin{proof}
We first prove the unconditional uniqueness. Let $T>0$ and $u, v$ be two $C_B-$mild solutions of \eqref{main}. Owing to \eqref{integral} and \eqref{SS}, we infer
\begin{equation*}
\|u(t)-v(t)\|_{L^\infty}\leq C\int\limits_0^t\,\left(t-\tau\right)^{\alpha/2}\,\|u(\tau)-v(\tau)\|_{L^\infty}\left(\|u(\tau)\|_{L^\infty}^{p-1}+\|v(\tau)\|_{L^\infty}^{p-1}\right)\,
d\tau.
\end{equation*}
This together with the singular Gronwall inequality (see Proposition \ref{SGI}) imply that $u=v$ on $[0,T]\times\R^N$. We turn now to the existence part. We use a fixed point argument. We introduce, for any $T, M>0$  the following complete metric space
$$
{\mathbf X}_{T, M}=\left\{\, u\in C_B([0,T]\times\R^N);\;\;\; \|u\|_T\leq M\,\right\},
$$
where $\|u\|_T=\displaystyle\sup_{0\leq t\leq T}\,\|u(t)\|_{L^\infty(\R^N)}$. Set
\begin{equation}
\label{Phi}
\Phi(u)(t)={\rm e}^{t\Delta}u_{0}+\int\limits_{0}^{t}{\rm e}^{(t-\tau)\Delta}\,\left(|\cdot|^{\alpha} |u(\tau)|^{p}\right)\,d\tau+\int\limits_{0}^{t}{\rm e}^{(t-\tau)\Delta}\,\left(\z(\tau)\,{\mathbf w}\right)\,d\tau.
\end{equation}
We will prove that the parameters  $T, M>0$ can be chosen so that  $\Phi$ is a contraction map from ${\mathbf X}_{T, M}$ into itself. Without loss of generality, we may assume that $T\leq 1$. Let $u\in {\mathbf X}_{T, M}$. Noticing that $-N\leq -3<-2<\alpha\leq 0$, $\sigma>-1$, and owing to \eqref{SS} and \eqref{ze}, we obtain that
\begin{eqnarray}
\label{Cont1}
\|\Phi(u)(t)\|_{L^\infty}&\leq& \|u_0\|_{L^\infty}+C\int\limits_0^t\, (t-\tau)^{\alpha/2}\, \|u(\tau)\|_{L^\infty}^p\,d\tau+\frac{t^{\sigma+1}}{\sigma+1}\|{\mathbf w}\|_{L^\infty},\\
\nonumber
&\leq& \|u_0\|_{L^\infty}+CM^p\frac{T^{1+\alpha/2}}{1+\alpha/2}+\frac{T^{\sigma+1}}{\sigma+1}\|{\mathbf w}\|_{L^\infty}.
\end{eqnarray}
Taking $M>\|u_0\|_{L^\infty}$ and choosing $T>0$ small enough, we easily deduce that $\Phi({\mathbf X}_{T,M})\subset {\mathbf X}_{T,M}$.
To show that $\Phi$ is a contraction we compute, for $u,v \in {\mathbf X}_{T,M}$,
\begin{eqnarray}
\label{Cont2}
\|\Phi(u)(t)-\Phi(v)(t)\|_{L^\infty}&\leq&C\int\limits_0^t\, (t-\tau)^{\alpha/2}\, \||u(\tau)|^p-|v(\tau)|^p\|_{L^\infty}\,d\tau\\
\nonumber
&\leq&C M^{p-1}\int\limits_0^t\, (t-\tau)^{\alpha/2}\, \|u(\tau)-v(\tau)\|_{L^\infty}\,d\tau\\
\nonumber
&\leq& CM^{p-1} T^{1+{\alpha/2}}\, \|u-v\|_T,
\end{eqnarray}
where we have used
$$
\left| |a|^p-|b|^p\right|\lesssim |a-b| \left(|a|^{p-1}+|b|^{p-1}\right).
$$
It follows that
\begin{equation}
\label{Cont3}
\|\Phi(u)-\Phi(v)\|_T\,\leq\,CM^{p-1} T^{1+{\alpha/2}}\, \|u-v\|_T.
\end{equation}
From \eqref{Cont3} we conclude that $\Phi$ is a contraction for $T>0$ sufficiently small. This finishes the existence part. The blowup criterion can be shown in a standard way by taking advantage of the fact that the local time of existence depends on $\|u_0\|_{L^\infty}$ for the choice $M=2\|u_0\|_{L^\infty}$.
Finally, since ${\rm e}^{t\Delta}$ preserve the positivity, we easily deduce that $u\geq 0$ provided that $u_0, {\mathbf w}\geq 0$. \qed\medskip
\end{proof}

Next, we turn to the case $\alpha>0$.

\begin{proof}
The proof of local existence follows from standard fixed point argument in a suitable complete metric space. To this end, we introduce
$$
{\mathbf Y}_{T, M}=\left\{\, u\in L^\infty([0,T]; {C}_\nu(\R^N));\;\;\; \sup_{0\leq t\leq T}\, \|u(t)\|_\nu\leq M\,\right\}.
$$
endowed with the distance
$$
d(u,v)=\|u-v\|_T=\sup_{0\leq t\leq T}\, \|u(t)-v(t)\|_\nu.
$$
Here $M>0$ and $0<T\leq 1$ to be fixed later. Define
\begin{eqnarray*}
{\mathbf I}(t)&=&{\rm e}^{t\Delta}u_{0},\\
{\mathbf J}(t)&=&\int\limits_{0}^{t}{\rm e}^{(t-\tau)\Delta}\,\left(|\cdot|^{\alpha} |u(\tau)|^{p}\right)\,d\tau,\\
{\mathbf K}(t)&=& \int\limits_{0}^{t}{\rm e}^{(t-\tau)\Delta}\,\left(\z(\tau)\,{\mathbf w}\right)\,d\tau,
\end{eqnarray*}
so that $\Phi(u)(t)={\mathbf I}(t)+{\mathbf J}(t)+{\mathbf K}(t)$ where $\Phi$ is given as in \eqref{Phi}. We will estimate separately the terms ${\mathbf I}(t)$, ${\mathbf J}(t)$ and ${\mathbf K}(t)$. First we compute
\begin{eqnarray*}
{\mathbf I}(t)(x)&=&\left(4\pi t\right)^{-{N/2}}\, \int\limits_{\R^N}\,{\rm e}^{-\frac{|x-y|^2}{4t}}\, u_0(y)\,dy,\\
&\leq &\|u_0\|_\nu\, \left(4\pi t\right)^{-{N/2}}\, \int\limits_{\R^N}\,{\rm e}^{-\frac{|x-y|^2}{4t}}\, \nu^{-1}(y)\,dy,\\
&=&\pi^{-{N/2}} \|u_0\|_\nu\, \int\limits_{\R^N}\,{\rm e}^{-|z|^2}\, \left(1+|x-2\sqrt{t}z|\right)^{-\frac{\alpha}{p-1}}\,dy.
\end{eqnarray*}
By \eqref{gl} with $\gamma=\frac{\alpha}{p-1}$, one obtains
\begin{equation}
\label{I}
\|{\mathbf I}(t)\|_\nu\leq C\|u_0\|_\nu,
\end{equation}
where $C$ is a positive constant depending only on $\alpha,\,p,\, N$.

Next, we compute (for $u\in {\mathbf Y}_{T,M}$)
\begin{eqnarray*}
{\mathbf J}(t)(x)&=&\int\limits_0^t\,\int\limits_{\R^N}\,\left(4\pi (t-\tau)\right)^{-{N/2}}\,{\rm e}^{-\frac{|x-y|^2}{4(t-\tau)}}|y|^{\alpha}\,|u(\tau,y)|^p\,dy\,d\tau,\\
&\leq&M^p\, \int\limits_0^t\,\int\limits_{\R^N}\,\left(4\pi (t-\tau)\right)^{-{N/2}}\,{\rm e}^{-\frac{|x-y|^2}{4(t-\tau)}}|y|^{\alpha}\,\nu(y)^{-p}\,dy\,d\tau,\\
&\leq& M^p\,\int\limits_0^t\,\left(\int\limits_{\R^N}\,\left(4\pi (t-\tau)\right)^{-{N/2}}\,{\rm e}^{-\frac{|x-y|^2}{4(t-\tau)}}|y|^{\alpha}\,(1+|y|)^{-\frac{\alpha}{p-1}}\,dy\right)\,d\tau,\\
&\leq& M^p\,\int\limits_0^t\, C \nu^{-1}(x)\, d\tau,
\end{eqnarray*}
where we have used \eqref{gl} in the last inequality. Therefore
\begin{equation}
\label{J}
\|{\mathbf J}(t)\|_\nu\leq C T M^P,
\end{equation}
where $C$ is a positive constant depending only on $\alpha,\,p,\, N$.

Let us now estimate the third term ${\mathbf K}$. Using again \eqref{gl} together with \eqref{ze}, we get
\begin{eqnarray*}
{\mathbf K}(t)(x)&=&\int\limits_0^t\,\int\limits_{\R^N}\,\left(4\pi (t-\tau)\right)^{-{N/2}}\,{\rm e}^{-\frac{|x-y|^2}{4(t-\tau)}}\z(\tau){\mathbf w}(y)\,dy\,d\tau,\\
&\leq&\|{\mathbf w}\|_\nu\, \int\limits_0^t\,\tau^{\sigma}\,\left(\int\limits_{\R^N}\,\left(4\pi (t-\tau)\right)^{-{N/2}}\,{\rm e}^{-\frac{|x-y|^2}{4(t-\tau)}}\,\nu^{-1}(y)\,dy\right)\,d\tau,\\
&\leq& C\frac{t^{1+\sigma}}{1+\sigma}\,\|{\mathbf w}\|_\nu\, \nu^{-1}(x).
\end{eqnarray*}
It follows that
\begin{equation}
\label{K}
\|{\mathbf K}(t)\|_\nu\leq C T^{\sigma+1}\|{\mathbf w}\|_\nu,
\end{equation}
where $C$ is a positive constant depending only on $\sigma,\,\alpha,\,p,\, N$.

Combining \eqref{I}-\eqref{J}-\eqref{K}, we end up with
\begin{equation}
\label{Stab}
\|\Phi(u)\|_T\leq C\|u_0\|_\nu+CTM^p+C T^{\sigma+1}\|{\mathbf w}\|_\nu.
\end{equation}
Choosing $M>C\|u_0\|_\nu$ and $T$ sufficiently small such that $C\|u_0\|_\nu+CTM^p+C T^{\sigma+1}\|{\mathbf w}\|_\nu\leq M$, we see that $\Phi({\mathbf Y}_{T,M})\subset {\mathbf Y}_{T,M}$.

Now we show that $\Phi$ is contractive. Let $u,v \in {\mathbf Y}_{T,M}$. Arguing as above, we obtain that
\begin{equation}
\label{Cont4}
\|\Phi(u)-\Phi(v)\|_T\leq CT M^{p-1}\, \|u-v\|_T\leq \frac{1}{2}\, \|u-v\|_T,
\end{equation}
for $M>C\|u_0\|_\nu$ and $T$ sufficiently small.  This enable us to conclude the proof of the existence part.

The uniqueness part follows easily from Lemma \ref{CPL}.

Finally, let us turn to regularity. Since $\z$ is continuous,  ${\mathbf w}\in {C}_{\nu}\subset {C}_B$ and $u\in L^\infty((0,T), {C}_B)$, standard regularity results for parabolic equations  (\cite[Appendix B]{QS} and \cite{LSU}) guarantee that $u$ is a classical solution. \qed\medskip
\end{proof}


\section{Nonexistence of global solutions}
We will focus in this section on blow-up results stated in Theorems \ref{Blow} - \ref{Bloww}.
\begin{proof}
Let $u$ be the maximal solution defined on $[0,T^*)\times\R^N$ and suppose that $T^*=\infty$. In order to obtain a contradiction we use the so-called test function method. See for instance \cite{BLZ, MP}. Let's choose two cut-off functions $f,g\in C^\infty([0,\infty)$ such that $0\leq f,g\leq 1$,
\begin{eqnarray}
\label{f}
 f(\tau)&=&\; \left\{
\begin{array}{cllll}1
\quad&\mbox{if}&\quad {1/2}\leq \tau\leq {2/3},\\\\ 0 \quad
&\mbox{if}&\quad \tau\in [0,{1/4}]\cup [{3/4},\infty),
\end{array}
\right.
\end{eqnarray}
and
\begin{eqnarray}
\label{g}
 g(\tau)&=&\; \left\{
\begin{array}{cllll}1
\quad&\mbox{if}&\quad 0\leq \tau\leq 1,\\\\ 0 \quad
&\mbox{if}&\quad \tau\geq 2.
\end{array}
\right.
\end{eqnarray}
For $T>0$, we introduce $\psi^{}_T(t,x)=f^{}_T(t)\,g^{}_T(x)$, where
\begin{eqnarray*}
f^{}_T(t)&=&\left(f\left(\frac{t}{T}\right)\right)^{\frac{p}{p-1}},\\\\
g^{}_T(x)&=&\left(g\left(\frac{|x|^2}{T}\right)\right)^{\frac{2p}{p-1}}.
\end{eqnarray*}
Multiplying both sides of the differential equation in \eqref{main} by $\psi^{}_T$ and integrating over $(0,T)\times\R^N$ we find
\begin{eqnarray}
\nonumber
&&\int\limits_0^T\,\int\limits_{\R^N}\,|x|^\alpha\,|u|^p\,\psi^{}_T\,dx\,dt+
\int\limits_0^T\,\int\limits_{\R^N}\,\z(t){\mathbf w}(x)\,\psi^{}_T\,dx\,dt+\int\limits_{\R^N}\,u_0(x)\psi^{}_T(0,x)\,dx\\
\label{psiT}
&=&-\int\limits_0^T\,\int\limits_{\R^N}\,u\,\Delta\psi^{}_T\,dx\,dt
-\int\limits_0^T\,\int\limits_{\R^N}\,u\,\partial_t\psi^{}_T\,dx\,dt,\\
\nonumber
&\leq&\int\limits_0^T\,\int\limits_{\R^N}\,|u|\,|\Delta\psi^{}_T|\,dx\,dt
+\int\limits_0^T\,\int\limits_{\R^N}\,|u|\,|\partial_t\psi^{}_T|\,dx\,dt.
\end{eqnarray}
Noticing that $f(0)=0$, we get
\begin{equation}
\label{u0}
\displaystyle\int\limits_{\R^N}\,u_0(x)\psi^{}_T(0,x)\,dx=0.
\end{equation}
 Next, applying Young inequality, we get
 \begin{eqnarray}
 \nonumber
 \int\limits_0^T\,\int\limits_{\R^N}\,|u(t,x)|\,|\Delta\psi^{}_T(t,x)|\,dx\,dt&=
 &\int\limits_0^T\,\int\limits_{\R^N}|x|^{\alpha/p}|u|\psi^{1/p}_T
 |x|^{-\alpha/p}|\Delta\psi^{}_T|\psi^{-1/p}_T\,dx\,dt,\\
 \label{Y1}
 &\leq&\frac{1}{2}\int\limits_0^T\,\int\limits_{\R^N}\,|x|^\alpha\,|u|^p\,\psi^{}_T\,dx\,dt\\
 \nonumber&+&
 C\int\limits_0^T\,\int\limits_{\R^N}\,|x|^{-\frac{\alpha}{p-1}}
 |\Delta\psi^{}_T|^{\frac{p}{p-1}}\,\psi^{-\frac{1}{p-1}}_T\,dx\,dt,\\
 \nonumber
 &\leq&\frac{1}{2}\int\limits_0^T\,\int\limits_{\R^N}\,|x|^\alpha\,|u|^p\,\psi^{}_T\,dx\,dt+
 CT^{1+\frac{N}{2}-\frac{p}{p-1}-\frac{\alpha}{2(p-1)}},
 \end{eqnarray}
 where we have used \eqref{f}, \eqref{g} and
 $$
 |\Delta\,g^{}_T(x)|\leq \frac{C}{T}\left(g\left(\frac{|x|^2}{T}\right)\right)^{\frac{2}{p-1}}.
 $$
 Similarly, we obtain that
 \begin{equation}
 \label{Y2}
 \int\limits_0^T\,\int\limits_{\R^N}\,|u|\,|\partial_t\psi^{}_T|\,dx\,dt\leq \frac{1}{2}\int\limits_0^T\,\int\limits_{\R^N}\,|x|^\alpha\,|u|^p\,\psi^{}_T\,dx\,dt+
 CT^{1+\frac{N}{2}-\frac{p}{p-1}-\frac{\alpha}{2(p-1)}}.
 \end{equation}
 Plugging estimates \eqref{psiT}, \eqref{u0}, \eqref{Y1} and \eqref{Y2} together, we find
 \begin{equation}
 \label{ww}
 \int\limits_0^T\,\int\limits_{\R^N}\,\z(t){\mathbf w}(x)\,\psi^{}_T(t,x)\,dx\,dt\,\leq \,C\,T^{1+\frac{N}{2}-\frac{p}{p-1}-\frac{\alpha}{2(p-1)}}.
 \end{equation}
 To conclude the proof we have to find a suitable lower bound of the left hand side in \eqref{ww}. For this purpose, we use \eqref{ze} to write (for $T\geq 2$)
 \begin{eqnarray}
 \nonumber
 \int\limits_0^T\,\int\limits_{\R^N}\,\z(t)\,{\mathbf w}(x)\,\psi^{}_T(t,x)\,dx\,dt&\geq&\int\limits_{{T/2}}^T\,\int\limits_{\R^N}\,\z(t)\,{\mathbf w}(x)\,\psi^{}_T(t,x)\,dx\,dt\\
 \nonumber
 &\geq&\left(\int\limits_{{T/2}}^T\,t^m\,f\left(\frac{t}{T}\right)^{\frac{p}{p-1}}\,dt\right)\;\left(\int\limits_{\R^N}\,{\mathbf w}(x) g^{}_T(x)\,dx\right)\\
 \label{LB}
 &\geq&C\, T^{m+1}\,\int\limits_{\R^N}\,{\mathbf w}(x) g^{}_T(x)\,dx.
 \end{eqnarray}
Since ${\mathbf w}\in L^1$ and $g^{}_T(x)\to g(0)=1$ as $T\to\infty$, we obtain by the dominated convergence theorem that
$$
\int\limits_{\R^N}\,{\mathbf w}(x) g^{}_T(x)\,dx\underset{T\to \infty}{\longrightarrow} \int\limits_{\R^N}\,{\mathbf w}(x)\,dx>0.
$$
Hence, for $T$ sufficiently large, we have
$$
\int\limits_{\R^N}\,{\mathbf w}(x) g^{}_T(x)\,dx\geq \frac{1}{2}\int\limits_{\R^N}\,{\mathbf w}(x)\,dx.
$$
Recalling \eqref{ww}, we end up with
\begin{equation}
\label{www}
\int\limits_{\R^N}\,{\mathbf w}(x)\,dx\,\leq\, C\,T^{\frac{N}{2}-m-\frac{p}{p-1}-\frac{\alpha}{2(p-1)}}.
\end{equation}
Noticing that $p<\frac{N-2m+\alpha}{N-2m-2}$ and letting $T$ to infinity in \eqref{www}, we get
$$
\int\limits_{\R^N}\,{\mathbf w}(x)\,dx\leq 0.
$$
This is obviously a contradiction and the proof is complete. \qed\medskip
\end{proof}
\begin{remark}
The main novelty in this proof is the new lower bound \eqref{LB} using only the parameter $m$ instead of $\sigma$.
This illustrate that the blow up depends on the behavior of $\z$ at infinity.
\end{remark}

We finally give the proof of Theorem \ref{Bloww}.

\begin{proof}
We employ the same argument as in the proof of Theorem \ref{Blow} with a different test function. For $\varepsilon>0$ small enough, we set
$$
\varphi^{}_T(t,x)=f^{}_T(t)\, g(\varepsilon\,|x|^2),
$$
where $f, g$ are given by \eqref{f}-\eqref{g}. Similar computations as above yield
\begin{equation}
\label{w4}
\int\limits_{\R^N}\,{\mathbf w}(x)\,dx\,\leq\, C\,\left(T^{-m}+T^{-m-\frac{p}{p-1}}\right).
\end{equation}
Noticing that $m>0$ and letting $T$ to infinity, we get
$$
\int\limits_{\R^N}\,{\mathbf w}(x)\,dx\leq 0.
$$
This finishes the proof of Theorem \ref{Bloww}. \qed\medskip
\end{proof}


\section{Global existence}
The following lemma will be useful in the proof of Theorem \ref{GE}.
\begin{lemma}
\label{GEL}
Let $N\geq 2$, $-1<\sigma<0$, $-2<\alpha<0$ and
\begin{equation}
\label{P}
p\geq 1+ \frac{2+\alpha}{N-2(\sigma+1)}.
\end{equation}
Then
\begin{equation}
\label{L1}
\frac{2+p\alpha}{Np(p-1)}<\frac{1}{p_c},
\end{equation}

\begin{equation}
\label{L2}
\frac{2+p\alpha}{Np(p-1)}<\frac{N+\alpha}{Np},
\end{equation}

\begin{equation}
\label{L3}
\frac{1}{p_c}+\frac{2\sigma}{N}<\frac{1}{p_c},
\end{equation}
and
\begin{equation}
\label{L4}
\frac{1}{p_c}+\frac{2\sigma}{N}<\frac{N+\alpha}{Np}.
\end{equation}

\end{lemma}
\begin{proof}
We only give the proof of \eqref{L4} since the other inequalities are trivial. Note that the inequality \eqref{L4} is equivalent to
\begin{equation}
\label{L5}
2\sigma\,p^2-(N-2+2\sigma)p+N+\alpha<0.
\end{equation}
To prove \eqref{L5}, consider the function
$$
\Theta(\tau)=2\sigma\,p^2-(\tau-2+2\sigma)p+\tau+\alpha=2\sigma\,p^2+\alpha-2(\sigma-1)p+(1-p)\tau.
$$
 Let $\tau^*=2\sigma+\frac{\alpha+2p}{p-1}$. It follows by using \eqref{P} that $N\geq \tau^*$. Since $p>1$ then $\Theta$ is a decreasing function. Hence
\begin{equation}
\label{Theta}
\Theta(\tau)\leq \Theta(\tau^*)=2\sigma (p-1)^2<0,\quad \forall\;\;\; \tau\geq \tau^*.
\end{equation}
Taking $\tau=N$ in \eqref{Theta}, we obtain \eqref{L5} as desired. \qed\medskip
\end{proof}

We turn now to the proof of Theorem \ref{GE}.

\begin{proof}
We argue as in \cite{YM}. Assume first that $p>\frac{N+\alpha}{N-2}$ (if $N>2$). Let $u_0\in L^{p_c}(\R^N)$, $\mathbf{w}\in L^{\ell}(\R^N)$ such that $\|u_0\|_{L^{p_c}}+\|\mathbf{w}\|_{L^{\ell}}<\varepsilon_0$ for some $\varepsilon_0>0$. We will show that the equation
\begin{equation}\label{eq:fixed-pt-eq}
u=w+\mathcal{F}(u) \hspace{0.2cm}\mbox{in}\hspace{0.2cm}\R^N\times (0,\infty); \quad w=\mathbf{S}_{0}(t)u_0+\int_0^t\mathbf{S}_{0}(t-s)\zeta(s)\mathbf{w} ds
\end{equation}
has a unique fixed point in some closed ball of $C_{b}((0,\infty);L^{p_c})$,where for $t>0$, \[\mathcal{F}(u)(t)=\int_{0}^{t}\mathbf{S}_{-\alpha}(t-s)|u(s)|^p ds,
\]
and $\mathbf{S}_{0}(t)={\rm e}^{t\Delta}$ is given by \eqref{S}.

Apply Proposition \ref{Key} to obtain the estimate
\begin{equation}\label{eq:self-mapping-bound}\|\mathcal{F}u\|_{L^{\infty}((0,\infty);L^{p_c})}\leq C\sup_{t>0}\|u(t)\|^{p}_{L^{p_c}}.
\end{equation}
Arguing similarly, we have that for $u$ and $v$ in $L^{\infty}((0,\infty);L^{p_c})$,
\begin{equation}\label{eq:contraction-bound}
\|\mathcal{F}(u)-\mathcal{F}(v)\|_{L^{\infty}(\mathbb{R}_{+};L^{p_c})}\leq C\|u-v\|_{L^{\infty}(\mathbb{R}_{+};L^{p_c})}
\left(\|u\|^{p-1}_{L^{\infty}(\mathbb{R}_{+};L^{p_c})}+\|v\|^{p-1}_{L^{\infty}(\R_{+};L^{p_c})}\right)).
\end{equation}
Using again Proposition \ref{Key}, we get
\begin{align*}
\|w\|_{L^{p_c}}&\leq \, C\|u_0\|_{L^{p_c}}+\bigg\|\int_0^t\mathbf{S}_{0}(t-s)\zeta(s)\mathbf{w}ds\bigg\|_{L^{p_c}}\\
&\leq\, C\|u_0\|_{L^{p_c}}+C\|\mathbf{w}\|_{L^{\ell}}\int_{0}^{1}\tau^{\sigma}(1-\tau)^{-(\sigma+1)}d\tau\\
&\leq\,C\|u_0\|_{L^{p_c}}+C\|\mathbf{w}\|_{L^{\ell}}\mathcal{B}(\sigma+1,-\sigma)	,	
\end{align*}
where $\mathcal{B}$ stands for the standard beta function. In the above inequalities, we have used $\ell>1$ together with the fact that $\frac{1}{p_c'}-\frac{1}{\ell'}=\frac{2(\sigma+1)}{N}$, $\ell'$ being the conjugate exponent of $\ell$. Therefore, \[\|w\|_{L^{p_c}}\leq C_1(\|u_0\|_{L^{p_c}}+\|\mathbf{w}\|_{L^{\ell}})\leq C_1\varepsilon_0.\]
 By choosing $R=2C_1\varepsilon_0$ and assuming  $2^p\,C\,C_1^{p-1}\,\varepsilon_0^{p-1}\leq 1/2$, we deduce from \eqref{eq:self-mapping-bound} and \eqref{eq:contraction-bound} that Equation \eqref{eq:fixed-pt-eq} has a unique fixed point in $\overline{B}_{R}(0)$.

 Next, we consider the case when  $1+ \frac{2+\alpha}{N-2(\sigma+1)} \leq p\leq \frac{N+\alpha}{N-2}$. Taking advantage of Lemma \ref{GEL}, we can pick a number $r>1$ such that
\begin{equation}\label{eq:cond-on-exponents}
\max\bigg\{\frac{\alpha p+2}{Np(p-1)},\frac{1}{p_c}+\frac{2\sigma}{N}\bigg\}<\frac{1}{r}<\min\bigg\{\frac{1}{p_c},\frac{N+\alpha}{Np}\bigg\},\quad r>p.
\end{equation}
In particular, we obtain that $1\leq \ell<p_c<r$. Define
$$
\mu=\frac{N}{2}\bigg(\frac{1}{p_c}-\frac{1}{r}\bigg).
$$
It follows that $0<\mu<\frac{1}{p}$. Introduce the function space $\mathbf{X}$ defined by
\[\mathbf{X}=\bigg\{u\in C_{b}\big((0,\infty);L^{r}(\R^N)\big);\;\;\;\; t^{\mu}u\in C_{b}\big((0,\infty);L^{r}(\R^N)\bigg\}
\]
 equipped with the distance $d(u,v)=\displaystyle\sup_{t>0}t^{\mu}\|u(t)-v(t)\|_{L^{r}}:=\|u-v\|_{\mathbf{X}}$. Clearly $(\mathbf{X}, d)$ is a complete metric space. We will show that  Eq. \eqref{eq:fixed-pt-eq} has a unique fixed point in some closed ball $\overline{B}_{K}(0)\subset \mathbf{X}$ with radius $K>0$ sufficiently small.
We estimate separately each of the terms of the right hand side of \eqref{eq:fixed-pt-eq}. Using  the smoothing estimate \eqref{SS} we obtain that
\begin{equation*}
\|\mathbf{S}_{0}(t)u_0\|_{L^{r}}\leq Ct^{-\frac{N}{2}\big(\frac{1}{p_c}-\frac{1}{r}\big)}\|u_0\|_{L^{p_c}}\leq Ct^{-\mu}\|u_0\|_{L^{p_c}}.
\end{equation*}
Using again \eqref{SS} together with the fact that  $r>\frac{Np}{N+\alpha}$, we get
\begin{align*}
\bigg\|\int_{0}^{t}\mathbf{S}_{\alpha}(t-s)|u(s)|^{p}ds\bigg\|_{L^{r}}&\leq C\int_{0}^{t}
(t-s)^{-\frac{N}{2}\big(\frac{p}{r}-\frac{1}{r}\big)+\frac{\alpha}{2}}\||u|^{p}\|_{L^{r/p}}ds\\
&\leq C\int_{0}^{t}\,(t-s)^{-\frac{N}{2}\big(\frac{p}{r}-\frac{1}{r}\big)+\frac{\alpha}{2}}s^{-p\mu}(s\|u\|_{L^{r}})^{p\mu}ds\\
&\leq C(\sup_{t>0}t^{\mu}\|u\|_{L^{r}})^{p}\int_{0}^{t}\,(t-s)^{-\frac{N(p-1)}{2r}+\frac{\alpha}{2}}s^{-p\mu}ds\\
&\leq Ct^{-\frac{N(p-1)}{2r}+\frac{\alpha}{2}-p\mu+1}\|u\|^{p}_{\mathbf{X}}
\int_{0}^{1}(1-s)^{-\frac{N(p-1)}{2r}+\frac{\alpha}{2}}s^{-p\mu}ds\\
&\leq Ct^{-\mu}\|u\|^{p}_{\mathbf{X}}\mathcal{B}\bigg(1-p\mu,1-\frac{N(p-1)}{2r}+\frac{\alpha}{2}\bigg)\\
&\leq Ct^{-\mu}\|u\|^{p}_{\mathbf{X}}.
\end{align*}
Concerning the last term, we use \eqref{SS} and proceed as follows
\begin{align*}
\bigg\|\int_{0}^{t}\mathbf{S}_{0}(t-s)\mathbf{w}\zeta(s)\bigg\|_{L^{r}}ds&\leq C\int_{0}^{t}(t-s)^{-\frac{N}{2}\big(\frac{1}{\ell}-\frac{1}{r}\big)+\frac{\alpha}{2}}\|\mathbf{w}\|_{L^{\ell}}
\zeta(s)ds\\
&\leq C\|\mathbf{w}\|_{L^{\ell}}t^{-\frac{N}{2}\big(\frac{1}{\ell}-\frac{1}{r}\big)+\sigma+1}
\int_{0}^{1}(1-s)^{-\frac{N}{2}\big(\frac{1}{\ell}-\frac{1}{r}\big)+\frac{\alpha}{2}}s^{\sigma}ds\\
&\leq C\|\mathbf{w}\|_{L^{\ell}}t^{-\frac{N}{2}\big(\frac{1}{\ell}-\frac{1}{r}\big)+\sigma+1}
\mathcal{B}\bigg(\sigma+1,1-\frac{N}{2}\big(\frac{1}{\ell}-\frac{1}{r}\big)+\frac{\alpha}{2}\bigg)\\
&\leq C\|\mathbf{w}\|_{L^{\ell}}\, t^{-\mu}.
\end{align*}
Put $\mathcal{G}u=w+\mathcal{F}(u), \hspace{0.1cm} w=\mathbf{S}_{0}(t)u_0+\int_0^t\mathbf{S}_{0}(t-s)\zeta(s)\mathbf{w}$, then  \[\displaystyle
\sup_{t>0}t^{\mu}\|\mathcal{G}u(t)\|_{L^{r}}\leq C\left(\|u_0\|_{L^{p_c}}+K^{p}+\|\mathbf{w}\|_{L^{\ell}}\right)\leq C\varepsilon_0.
\]
 By taking $\varepsilon_0$ and $K>0$ sufficiently small, one can deduce that $C(\|u_0\|_{L^{p_c}}+\|\mathbf{w}\|_{L^{\ell}})\leq K$. Hence $\displaystyle
\sup_{t>0}t^{\mu}\|\mathcal{G}(t)\|_{L^{r}}\leq K$. Therefore $\mathcal{G}$ maps $\overline{B}_{K}(0)$ into itself. Arguing in a similar fashion, we can show that $\mathcal{G}$ is a contraction map on $\overline{B}_{K}(0)$ for sufficiently small $K$. The Banach fixed point theorem allows us to deduce the existence of an unique solution $u$ to \eqref{integral} in $\overline{B}_{K}(0)\subset\mathbf{X}$ for sufficiently small $K$. This finishes the proof of Theorem \ref{GE}. \qed\medskip
\end{proof}

\section*{Acknowledgements} The author warmly thanks the anonymous referee for his/her useful and nice comments that were very important to improve the paper. The author is grateful to Professor Philippe Souplet for providing him some useful references.



\begin{thebibliography}{20}


\bibitem{AG} {\sc J. Aguirre and J. Giacomoni}, {\it The shape of blow-up for a degenerate parabolic equation}, Differ. Integral Equ., {\bf 14} (2001),  589--604.



\bibitem{BTW}{\sc B. Ben Slimene, S. Tayachi and F. B. Weissler}, {\it Well-posedness, global existence and large time behavior for Hardy-H\'enon parabolic equations}, Nonlinear Analysis, 152 (2017), 116--148.


\bibitem{BLZ} {\sc C. Bandle, H. A. Levine and Qi S. Zhang}, {\it Critical Exponents of Fujita Type for Inhomogeneous Parabolic Equations and Systems}, Journ. of Math. Anal. and App., 251 (2000), 624--648.

\bibitem{DL} {\sc K. Deng and H. A. Levine}, {\it The role of critical exponents in blowup theorems, the sequel}, J. Math. Anal. Appl., {\bf 243} (2000), 85--126.



\bibitem{YM}  {\sc G. Diebou and M. Majdoub}, {\it Blow up and global existence of solutions to a Hardy-H\'{e}non equation with a spatial-temporal forcing term}, {\tt Submitted}.


\bibitem{DM1986} {\sc J. Dixon and S. Mckef}, {\it Weakly singular discrete Gronwall inequalities}, Z. angew. Math. Mech., {\bf 64} (1986), 535--544.


\bibitem{EG} {\sc M. J. Esteban and J. Giacomoni}, {\it Existence of global branches of positive solutions for semilinear elliptic degenerate problems}, J. Math. Pures Appl., {\bf 79} (2000), 715--740.

\bibitem{fujita}{\sc H. Fujita}, {\it On the blowing up of solutions of the Cauchy problem for $u_t=\Delta u+u^{1+\alpha}$}, J. Fac. Sci. Univ. Tokyo Sec. IA Math. {\bf 13} (1966), 109--124.

\bibitem{GV} {\sc V. A. Galaktionov and J. L. V\'azquez}, {\it The problem of blow-up in nonlinear parabolic equations}, Discrete Contin. Dyn. Syst., {\bf 8} (2002), 399--433.

\bibitem{Gia} {\sc J. Giacomoni}, {\it Some results about blow-up and global existence to a semilinear degenerate heat equation}, Rev. Mat. Complut., {\bf 11} (1998), 325--351.

    \bibitem{Han} {\sc Y. Han}, {\it Blow-up phenomena for a reaction diffusion equation with special diffusion process}, Applicable Analysis (2020), DOI: 10.1080/00036811.2020.1792447

\bibitem{Hayak} {\sc K. Hayakawa}, {\it On nonexistence of global solutions of some semilinear parabolic differential equations}, Proc. Japan Acad. {\bf 49} (1973), 503--505.

 \bibitem{Hu} {\sc B. Hu},  {\it Blow Up Theories for Semilinear Parabolic Equations}, Springer, Berlin (2011).

\bibitem{JKS} {\sc M. Jleli, T. Kawakami and B. Samet}, {\it Critical behavior for a semilinear parabolic equation with forcing term depending of time and space}, J. Math. Anal. Appl. {\bf 486} (2020), 123931.


\bibitem{LSU} {\sc O. A. Ladyzenskaja, V. A. Solonnikov and N. N. Ural\'ceva}, {\it Linear and quasilinear
equations of parabolic type}, Amer. Math. Soc., Transl. Math. Monographs, Providence, R.I.(1968).


\bibitem{Le} {\sc H. A. Levine}, {\it The role of critical exponents in blowup theorems}, SIAM Rev., {\bf 32} (1990), 269--288.

\bibitem{LM} {\sc H. A. Levine and P. Meier}, {\it The value of the critical exponent for reaction-diffusion equations in
cones},  Arch. Rational Mech. Anal., {\bf 109} (1989), 73--80.

\bibitem{MT1} {\sc A. V. Martynenko and A. F. Tedeev}, {\it Cauchy problem for a quasilinear parabolic equation with a source term and an inhomogeneous density}, Comput. Math. Math. Phys., {\bf 47} (2007), 238--248.

\bibitem{MT2} {\sc A. V. Martynenko and A. F. Tedeev}, {\it On the behavior of solutions to the Cauchy problem for a degenerate parabolic equation with inhomogeneous density and a source}, Comput. Math. Math. Phys., {\bf 48} (2008), 1145--1160.


\bibitem{MP} {\sc E. Mitidieri and S. I. Pohozaev}, {\it A priori estimates and blow-up of solutions of nonlinear partial differential equations and inequalities}, Proc. Steklov Inst. Math., {\bf 234} (2001), 3--383.

\bibitem{Stu} {\sc C. A. Stuart}, {\it A critically degenerate elliptic Dirichlet problem, spectral theory and bifurcation},  Nonlinear Analysis, {\bf 190} (2020),  Article ID 111620.

\bibitem{Tan} {\sc Z. Tan}, {\it Reaction-diffusion equations with special diffusion processes}, Chin. J. Contemp. Math., {\bf 22} (2001), 371--382; {\tt translation from} Chin. Ann. Math., Ser. A, {\bf 22} (2001), 597--607.


\bibitem{Qi} {\sc Yuan-wei Qi}, {\it The critical exponents of parabolic equations and blow-up in $\R^n$}, Proceedings of the Royal Society of Edinburgh, {\bf 128A} (1998), 123--136.


\bibitem{QS} {\sc P. Quittner and P. Souplet}, {\it Superlinear parabolic problems}, {Birkh\"auser Verlag, Basel} (2007), {xii+584}.

\bibitem{Tay} {\sc S. Tayachi}, {\it Uniqueness and non-uniqueness of solutions for critical Hardy-H\'enon parabolic equations}, J. Math. Anal. Appl., {\bf 488} (2020), 123976.

\bibitem{Zh1} {\sc Q. S. Zhang}, {\it A new critical phenomenon for semilinear parabolic problem}, J. Math. Anal. Appl., {\bf 219} (1998), 123--139.

\bibitem{Zh2} {\sc Q. S. Zhang}, {\it Blow up and global existence of solutions to an inhomogeneous parabolic system}, J. Differential Equations, {\bf 147} (1998), 155--183.

\bibitem{Wang} {\sc X. Wang}, {\it On the Cauchy problem for reaction-diffusion equations}, Transactions of the American Mathematical Society, 337 (1993), 549--590.


\end{thebibliography}
\end{document}